\documentclass{amsart}
\usepackage[all]{xy}
\usepackage{amssymb}

\usepackage[OT2,T1]{fontenc}
\DeclareSymbolFont{cyrletters}{OT2}{wncyr}{m}{n}
\DeclareMathSymbol{\sha}{\mathalpha}{cyrletters}{"58}

\newcommand{\F}{\mathbf{F}}

\newcommand{\Z}{\mathbf{Z}}
\newcommand{\Q}{\mathbf{Q}}
\newcommand{\R}{\mathbf{R}}
\newcommand{\C}{\mathbf{C}}
\newcommand{\p}{\mathfrak{p}}
\newcommand{\bm}{\boldsymbol\mu }

\DeclareMathOperator{\Gal}{Gal}
\DeclareMathOperator{\tor}{tor}
\DeclareMathOperator{\coker}{coker}
\DeclareMathOperator{\rk}{rk}
\DeclareMathOperator{\ns}{ns}
\DeclareMathOperator{\nd}{nd}
\newtheorem{lemma}{Lemma}[section]
\newtheorem{proposition}[lemma]{Proposition}

\newtheorem{conjecture}[lemma]{Conjecture}

\theoremstyle{definition}
\newtheorem{remark}[lemma]{Remark}

\begin{document}
\title[On the density of abelian surfaces with $\# \sha = 5n^2$]{On the density of abelian surfaces with Tate- Shafarevich group of order five times a square}
\author{Stefan Keil}
\address{Institut f\"ur Mathematik, Humboldt-Universit\"at zu Berlin, Unter den Linden 6, D-10099 Berlin, Germany}
\email{keil@math.hu-berlin.de}
\author{Remke Kloosterman}
\address{Institut f\"ur Mathematik, Humboldt-Universit\"at zu Berlin, Unter den Linden 6, D-10099 Berlin, Germany}
\email{klooster@math.hu-berlin.de}
\thanks{}
\begin{abstract}
Let $A=E_1\times E_2$ be be the product of two elliptic curves over $\Q$, both 
having a rational five torsion point $P_i$. Set $B=A/\langle 
(P_1,P_2)\rangle$. In this paper we give an algorithm to decide whether 
the Tate-Shafarevich group of the abelian surface $B$ has square order or order five times a square, assuming that we can find a basis for the Mordell-Weil groups of both $E_i$, and that the Tate-Shafarevich groups of the $E_i$ are finite.

We considered all pairs $(E_1,E_2)$, with prescribed bounds on the conductor and the coefficients on a minimal Weierstrass equation. In total we considered around $20.0$ million of abelian surfaces of which $49.16\%$ have a Tate-Shafarevich group of non-square order. 
\end{abstract}
\date{\today}
\thanks{The first author is supported by a scholarship from the Berlin Mathematical School (BMS). Both authors thank Tom Fisher for pointing out the method ``MordellWeilShaInformation()'' in MAGMA, and both authors thank the referees for their comments and suggestions}
\maketitle
\section{Introduction}
Let $A$ be an abelian variety over a number field $K$. Then the Tate-Shafarevich group $\sha(A/K)$ plays an important role in understanding the arithmetic of $A$. For example, it contains information on the tightness of the upper bound on the Mordell-Weil rank obtained by $m$-descent. Moreover the order of this group, which is conjectured to be finite, plays a role in the Birch and Swinnerton-Dyer conjecture.

The Tate-Shafarevich group comes with a pairing, the so-called Cassels-Tate pairing, which depends on the choice of a polarization $\lambda: A\to A^\vee$:
\[ \langle \cdot,\cdot \rangle_{\lambda}: \sha(A/K) \times \sha(A/K) \to \Q/\Z.\]
Let $\sha(A/K)_{\nd}$ denote the Tate-Shafarevich group modulo its maximal divisible subgroup. If $\lambda$ is an isomorphism, i.e., $A$ is principally polarized, then the induced pairing on $\sha(A/K)_{\nd}$ is non-degenerate. If moreover this pairing is alternating, then for all primes $p$ the cardinality of the $p$-divisible part $\sha(A/K)_{\nd}[p^\infty]$ is a perfect square, thus if $\sha(A/K)$ is finite, its order is a perfect square.

Tate \cite{Tate} showed that if $\lambda$ is an isomorphism and is also induced from a $K$-rational divisor on $A$ then the Cassels-Tate pairing is actually alternating, as for example for elliptic curves. However, if $\dim A>1$ then $A$ may not admit a principal polarization and even when $A$ is principally polarized then this polarization need not to be induced by a $K$-rational divisor on $A$. Poonen and Stoll \cite{PS} in fact showed that there exist genus $2$ curves $C$ such that $\# \sha(J(C))$ is twice a square. Moreover, they showed that if one assumes that $\sha(J(C))$ is finite for all genus 2 curves $C/\Q$ then the density of the Jacobians of genus $2$ curves that have non-square order Tate-Shafarevich groups exists, and they showed that it is about $13\%$. 

For arbitrary abelian varieties Flach \cite{Flach} showed that if $\#\sha(A/K)=kn^2$, with $k$ square free, then $k$ divides $2$ times the degree of any polarization on $A$. Hence for principally polarized abelian varieties one has that $\#\sha(A/K)$ is either a square or twice a square, if it is finite, but for general abelian varieties there are more possibilities. Stein \cite{Stein} constructed for every prime number $p<25,000$ an example of a $p-1$-dimensional abelian variety $A_p/\Q$ such that $\#\sha(A_p)=pn^2$. 

We restrict now to the case of $\dim A=2$. The constructions of Poonen-Stoll and of Stein yield examples of abelian surfaces such that $\#\sha(A/K)$ is a square, twice a square or three times a square. One might wonder which further possibilities occur. Recently, the first author \cite{Keil} showed that there exist abelian surfaces such that the Tate-Shafarevich group has order five times a square and seven times a square.

In this paper we will have a closer look on the construction of abelian surfaces with Tate-Shafarevich group of order five times a square. The examples of \cite{Keil} are members of a two dimensional family of abelian surfaces with a polarization of degree $5^2$. Moreover, one can show that for a general member of this family every polarization it possesses has degree a multiple of $5$, thus they are not a priori excluded by Flach's theorem and might have a Tate-Shafarevich group of order five times a square. 

The construction of this family goes as follows. Let $(E,O)$ be an elliptic curve over $\Q$ with a point $P$ of order $5$, then there exists a $d\in \Q^*$ such that $((E,O),P)$ is isomorphic to $((E_d,O),(0,0))$ with
\[ E_d:y+(d+1)xy+dy=x^3+dx^2.\]
Take now two numbers $d_1,d_2\in \Q^*$ and consider $B_{d_1,d_2}:=E_{d_1}\times E_{d_2}/\langle(0,0)\times (0,0) \rangle$. Then  $A_{d_1,d_2}:=E_{d_1}\times E_{d_2}\to B_{d_1,d_2}$ is an isogeny of degree 5. Moreover, if the two elliptic curves are not isogenous, then all polarization on $B_{d_1,d_2}$ have degree divisible by $5$. The $B_{d_1,d_2}$'s are the family we consider. In our case we know that $\sha(A_{d_1,d_2}/\Q)$ has square order, if it is finite, since it is isomorphic to the product of the two Tate-Shafarevich groups of $E_{d_1}$ and $E_{d_2}$. 

The behavior of the Tate-Shafarevich group under isogenies is well-known. This behavior is part of Tate's proof of the invariance of the Birch and Swinnerton-Dyer conjecture; for more on this see Section~\ref{SecPre}. The upshot of this is the following: Let $\varphi:A\to B$ be an isogeny and assume that either $\#\sha(A/K)$ or $\#\sha(B/K)$ is finite (which implies that both are finite). Denote by $\varphi^\vee:B^\vee \to A^\vee$ the dual isogeny. For a field $L\supset K$ denote by $\varphi_L:A(L)\to B(L)$ the induced map on $L$-rational points. Let $S$ be a finite set of places containing the primes where $A$ has bad reduction, the infinite places and the primes dividing the degree of $\varphi$. Then the following holds:
 \[ \frac{\# \sha(A/K)}{\#\sha(B/K)}=   \frac{\# \ker \varphi_K \#\coker \varphi_K^{\vee}}{\# \ker \varphi_K^{\vee} \#\coker \varphi_K}  \prod_{v\in S}  \frac{\#\coker \varphi_{K_v}}{\#\ker \varphi_{K_v}}. \]
In Sections~\ref{SecGlo} and~\ref{SecLoc} we show that for our choice of abelian surfaces the above mentioned cardinalities of kernels and co-kernels can be determined, provided one has a basis for the Mordell-Weil group of both $E_{d_1}$ and $E_{d_2}$. 
(Actually something weaker is enough, see end of Section~\ref{SecGlo}.) Hence, given a basis for the Mordell-Weil groups of both elliptic curves we can determine whether $\#\sha(B/\Q)$ is a square or a non-square.  

For all pairs $(d_1,d_2)$ with $d_i=u_i/v_i$, such that $\max(|u_i|,|v_i|)$ is bounded by $N=50,000$ and the conductor of $E_{d_i}$ is bounded by $C=10^6$, we computed this product of cardinalities of kernels and cokernels. There are $2,445,366$ such pairs and $47.01\%$ of these surfaces have a Tate-Shafarevich group of non-square order. Also we computed these cardinalities for all pairs $(d_1,d_2)$, such that the absolute value of the numerator and denominator of $d_i$ is bounded by $N=100$. There are $18,522,741$ of such pairs and $49.31\%$ of them have a Tate-Shafarevich group of non-square order. Based on our computations we expect that the density of abelian surfaces $B_{d_1,d_2}$ with non-square Tate-Shafarevich groups exists and is around $50\%$. For some heuristics see the end of the last section.

The outline of this paper is as follows. In Section~\ref{SecPre} we discuss some preliminaries and in Section~\ref{SecCon} we explain in more detail the construction of the considered familiy of abelian surfaces. In Section~\ref{SecGlo} we discuss how we can calculate the global quotient and which conditions on $E_{d_1}$ and $E_{d_2}$ are needed for this. In Section~\ref{SecLoc} we discuss how we calculate the local quotient, which turns out to be much easier computationally. In Section~\ref{SecAlg} we sketch the algorithm used for the computations of the densities and finally in Section~\ref{SecRes} we discuss the obtained results.

\section{Preliminaries}\label{SecPre}
Let $K$ be a number field, and let $G_K$ be the absolute Galois group $\Gal(\overline{K}/K)$. For a (finite or infinite) place $v$ of $K$ denote by $K_v$ its completion with respect to $v$ and $G_{K_v}$ its absolute Galois group.

Let $A/K$ be an abelian variety. Denote by $A^\vee$ the dual abelian variety.
Then the Tate-Shafarevich group of $A/K$ is defined as 
\[\sha(A/K) := \ker \left( H^1(G_K,A)\to \prod_v H^1(G_{K_v},A) \right),\] where the product is taken over all finite and infinite places of $K$. Let $\varphi :A\to B$ be an isogeny of abelian varieties, then the $\varphi$-Selmer group of $A/K$ is defined as
\[S^{\varphi}(A/K):=\ker \left( H^1(G_K,A[\varphi]) \to \prod_v H^1(G_{K_v},A) \right).\]
Here $[\cdot]$ means ``kernel of''.

The Tate-Shafarevich group is torsion. It is conjectured to be finite and the $\varphi$-Selmer group is known to be finite. The $m$-torsion subgroup of the Tate-Shafarevich group fits in an exact sequence
\[ 0\to A(K)/mA(K) \to S^{[m]}(A/K) \to \sha(A/K)[m]\to 0. \]
I.e., it measures the difference between the $m$-Selmer group and $A(K)/mA(K) $. In theory the $m$-Selmer group is computable, hence the Tate-Shafarevich group measures the difference between the upper bound on the Mordell-Weil rank obtained by doing $m$-descent and the actual Mordell-Weil rank of $A$.

The Tate-Shafarevich group plays also a role in the Birch and Swinnerton-Dyer conjecture:
\begin{conjecture}[Birch and Swinnteron-Dyer] Let $A/K$ be an abelian variety and $L(A,s)$ its $L$-series. Set $r:=\rk A(K)$. Then $\sha(A/K)$ is finite, $L(A,s)$ has a zero of exact order $r$ at $s=1$, and
\[ \lim_{s\to 1} \frac{L(A,s)}{(s-1)^r} = \frac{2^r \#\sha(A/K) R_A \prod\int_{A(K_v)} |\omega|_v}{\#A(K)_{\tor} \#A^{\vee}(K)_{\tor}}. \]
\end{conjecture}

The left hand side of this conjecture is invariant under isogeny. Cassels \cite{Cassels} ($\dim A=1$) and Tate \cite{Tate} ($\dim A\geq 1$) proved that the right hand side is also invariant under isogeny. I.e., if $\varphi:A \to B$ is an isogeny then
\[ \frac{\#\sha(A/K)}{\# \sha(B/K)} = \frac{R_B \#A(K)_{\tor} \#A^\vee(K)_{\tor} \prod\int_{B(K_v)}  |\omega|_v}{R_A \#B(K)_{\tor} \#B^\vee(K)_{\tor} \prod\int_{A(K_v)} |\omega|_v }.\]
This formula was used by Schaefer and the second author \cite{KloSche} to provide examples of elliptic curves with large Selmer groups, by Matsuno \cite{Matsuno} and by the second author \cite{KloSha} to provide examples of elliptic curves with large Tate-Shafarevich groups and by Flynn and Grattoni \cite{FlyGra} to compute several Selmer groups.

However, for calculation purposes the right hand side is not suitable. One can rewrite the right hand side as follows: For a field $L\supset K$ let $\varphi_L$ denote the group homomorphism $\varphi_L:A(L)\to B(L)$. Then
 \[ \frac{\# \sha(A/K)}{\#\sha(B/K)}=   \frac{\# \ker \varphi_K \#\coker \varphi_K^{\vee}}{\# \ker \varphi_K^{\vee} \#\coker \varphi_K}  \prod_{v}  \frac{\#\coker \varphi_{K_v}}{\#\ker \varphi_{K_v}} .\]
We will call the first factor with the $\varphi_K$ the {\em global factor}, the second factor with the $\varphi_{K_v}$ the {\em local factor}. If $v$ is a finite prime of good reduction and $v$ does not divides the degree of the isogeny then $\#\coker \varphi_{K_v}=\#\ker \varphi_{K_v}$, hence the product on the right hand side is a finite product, where only the bad primes, the infinite primes and the primes dividing the degree of the isogeny are taken into account.

It is known that if the analytic rank of an elliptic curve is at most 1, then its Tate-Shafarevich group is finite and the analytic rank equals the Mordell-Weil rank; otherwise we will assume these two conjectures.

\section{Constructing a family of abelian surfaces}\label{SecCon}
We will construct a two-dimensional family of abelian surfaces $B/K$, whose members are quotients of products of two elliptic curves $E_1,E_2$ by an isogeny of degree $5$. Therefore $\# \sha(B/K) \cdot 5^a=\# \sha(E_1\times E_2)$, for an $a \in \Z$. Since $\# \sha(E_1\times E_2)$ is a square it follows that $\# \sha(B/K)$ modulo squares is one of $\{1,5\}$. Additionally we have that for a general member of this family every polarization has degree divisible by $5$. Thus Flach's theorem does not restrict us further.

Let $G/K$ be a group scheme of prime order $\ell$. Let $E_1,E_2/K$ be two elliptic curves such that $G$ is a subgroup scheme of both $E_1$ and $E_2$. Let $A=E_1\times E_2$ and $B=A/G$. Then $\varphi:A\to B$ has degree $\ell$. Moreover, one can show that either $E_1$ and $E_2$ are isogenous or every polarization on $B$ has degree a multiple of $\ell$. Hence for general $E_1,E_2$ we are in the second case.

Consider the case $G=\Z/\ell\Z$, i.e., $G$ is generated by a $K$-rational point. Since for $\ell>4$ the functor $Y_1(\ell)$ is representable one has a universal family of elliptic curves $E$ with a point $P$ of order $\ell$. In the case $\ell=5$ the universal family is given by 
\[E_d: y^2+(d+1)xy+dy=x^3+dx^2,\ P=(0,0),\]
for any $d\in K^*$ with $d^2+11d-1\neq 0$. The four non-trivial $5$-torsion points are $(0,0), (-d, d^2), (-d, 0), (0,-d)$.  If we move $(0,-d)$ to $(0,0)$ and bring the curve in standard form we obtain $E_d$. If we move $(-d,d^2)$ or $(-d,0)$ to $(0,0)$ and bring the elliptic curve in standard form then we obtain $E_{-1/d}$.

We restrict now to the case $K=\Q$, $\ell = 5$, and $G$ is generated by a $\Q$-rational point. Fix $d_1$ and $d_2$ in $\Q^*$, set $A:=E_{d_1}\times E_{d_2}$. The rational $5$-torsion subgroup of $A$ has four diagonally embedded subgroups of order $5$. Let $G=\Z/5\Z$ be one of those, i.e., it is the subscheme of $A$ generated by $(0,0)\times [n](0,0)$, with $n\in \{1,2,3,4\}$. Let $B:=A/G$. Then $B$ is a candidate for an abelian surface such that $\sha(B/\Q)$ has order five times a square. To actually check whether $\sha(B/\Q)$ has non-square order we will now calculate both the local and the global factor. 

Note, that the $16$ surfaces $B/\Q$ one obtains by replacing $d_i$ by $-1/d_i$ and using the four values of $n$ break into two pairs of $8$ isomorphic surfaces. For fixed $d_1,d_2$ the surfaces corresponding to $n=1,4$ lie in one of these isomorphism classes and those for $n=2,3$ in the other one. We will see in the next two sections that for fixed $d_1,d_2$ the size of $\sha(B/\Q)$ is independent of $n$, thus all $16$ surfaces will have Tate-Shafarevich groups of same cardinality. Therefore, for the computations we will only consider the case $d_1,d_2>0$ and $n=1$.

Let $A'$ be the quotient of $E_{d_1}\times E_{d_2}$ by $\langle (0,0)\times O,O\times (0,0)\rangle$ and $E'_{d_i}$ be the quotient of $E_{d_i}$ by $\langle (0,0)\rangle$. The isogeny $A\to A'$ factors as $A\to B \to A'$. Consider now the dual picture
\[ (A')^\vee \to B^\vee \to A^\vee.\]
Since $A$ and $A'$ are products of elliptic curves, they are principally polarized. Therefore we have the following factorization
\[ A'\to B^\vee \to A.\]
The kernel of $A'\to A$ is Cartier dual to the kernel of $A\to A'$, and hence is isomorphic to $(\bm_5)^2$. The kernel of $A'\to B^\vee$ is isomorphic to $\bm_5$ embedded with $(1,-n)$ in $(\bm_5)^2$. 

Summarizing we have the following diagram:
\[ \xymatrix{
&B\ar[rd] ^\psi\\
A=E_{d_1}\times E_{d_2} \ar[ur]^\varphi  \ar@/^1pc/[rr]^{\rho=\eta_1\times \eta_2}&   &A'=E_{d_1}'\times E_{d_2}'\ar[ld]^{\psi^\vee}  \ar@/^1pc/[ll]^{\rho^\vee}\\
& B^\vee \ar[lu]^{\varphi^\vee}
}\]

\begin{lemma} Suppose $L=\Q$. Then $\ker \varphi_{\Q}\cong \Z/5\Z$ and $\ker \varphi^{\vee}_{\Q}=0$.
\end{lemma}
\begin{proof} Since $A[\varphi]=\Z/5\Z$ it follows that $A'[\varphi^\vee]=\bm_5$. Taking $\Q$-rational points yields the lemma.
\end{proof}

\begin{lemma} Suppose $L=\R$. Then $\ker \varphi_{\R}\cong \Z/5\Z$ and $\coker \varphi_{\R}= 0$.
\end{lemma}

\begin{proof}
The first assertion is automatic. The non-trivial element in $\Gal (\C / \R)$ acts on the fiber of an element of $B(\R)$ under $\varphi_\C$ either by swaping elements or fixing them. Since the degree of $\varphi$ is not divisible by $2$ at least one element in the fiber is fixed, hence lies in $A(\R)$.
\end{proof}

Let $S$ be the set of primes where $A$ has bad reduction, together with $5$.
  Using the above lemmas it follows that
 \[ \frac{\# \sha(A/\Q)}{\#\sha(B/\Q)}= \frac{\#\coker \varphi_{\Q}^{\vee}}{\# \coker \varphi_{\Q}}  \prod_{v\in S}  \frac{\#\coker \varphi_{\Q_v}}{\#\ker \varphi_{\Q_v}} . \]
 In the next two sections we will first explain how to determine the first factor, the {\em global factor}, then how to determine the second factor, the {\em local factor}.

\section{Determining the global factor}\label{SecGlo}
To determine  $ \frac{\#\coker \varphi^{\vee}_\Q}{\# \coker \varphi_\Q}$ we assume for the moment that one has a basis for the Mordell-Weil groups $E_{d_1}(\Q), E_{d_2}(\Q),E'_{d_1}(\Q)$ and $E'_{d_2}(\Q)$. 
We will now explain how one can determine $\coker \varphi_\Q$ and $\coker \varphi^\vee_\Q$ from this.

Using $\rho^\vee=\varphi^\vee \circ \psi^\vee$ we obtain a surjective homomorphism $\coker \rho^\vee_\Q \to \coker \varphi^\vee_\Q$. With Hilbert's Theorem 90 we obtain 
\[H^1(G_\Q,A'[\rho^\vee])=H^1(G_\Q,\bm_5^2) = (\Q^*/\Q^{*5})^2, \]
\[H^1(G_\Q,B^\vee[\varphi^\vee])=H^1(G_\Q,\bm_5)=\Q^*/\Q^{*5}.\] 
The surjection $\coker \rho^\vee_\Q \to \coker \varphi^\vee_\Q$ becomes $(x,y)\mapsto x^n/y$ as the map from $(\Q^*/\Q^{*5})^2$ to $\Q^*/\Q^{*5}$. 
One sees immediately that the image of this map is independent of $n$, hence to compute $\coker \varphi^\vee_\Q$ we may set $n=1$. In order to determine $\coker \varphi^\vee_\Q$ it suffices to determine a basis for both $\coker \eta^\vee_{i,\Q}$ in $\Q^*/\Q^{*5}$. This can be done quite easily following \cite[Exercise 10.1]{SilvermanI}: Suppose that $f$ is a function on $E_{d_i}$ with divisor $5(0,0)-5O$. Then there exists a unique constant $c \in \Q^*/\Q^{*5}$ such that the map
\[ \coker \eta^\vee_{i,\Q} \to \Q^*/\Q^{*5},\]
sending $P \not = (0,0),O$  to $cf(P) \bmod \Q^{*5}$, is a well-defined and injective group homomorphism and its image agrees with the image of the natural embedding of $\coker \eta^\vee_{i,\Q}$ into $H^1(G_\Q,E'_{d_i}[\eta_i^\vee]) \cong \Q^*/\Q^{*5}$. In our case we can take the function $f= -x^2+y+xy$ and the constant $c=1$. The point $(0,0)$ is mapped to $d^{-1}$ and $O$ to $1$ by linearity.

An element of $\Q^*/\Q^{*5}$ is determined by the valuations at each prime. Write now $d=u/v$ and let $S$ be the set of all primes $p$ dividing five times the minimal discriminant of $E_{d}$, i.e.,  $p\mid 5uv(u^2+11uv-v^2)$. Define $$\Q(S,5):=\{x \in \Q^*/\Q^{*5} \mid \ v_p(x)\equiv 0 \bmod 5, \ \forall p \notin S\}.$$
From the same exercise from \cite{SilvermanI} it follows that $f(\coker \eta^\vee_\Q) \subset \Q(S,5)$. Hence we can represent an element of $\coker \eta^\vee_{\Q}$ by its valuation at each prime number $p\in S$. Once the cokernels of both $\eta^\vee_{i,\Q}$ are established, the cokernel of $\varphi^\vee_\Q$ can be computed easily.

To determine the cokernel of $\varphi_\Q$ we use the following exact sequence
\[ 0 \to \ker (\psi_\Q)/\varphi(\ker \rho_\Q)\to \coker \varphi_\Q \stackrel{\psi}{\to} \coker \rho_\Q \to \coker \psi_\Q \to 0.\]
Note that  $\ker(\psi_\Q)=\varphi(\ker \rho_\Q)$. Set $K:=\Q(\zeta_5)$, where $\zeta_5$ is a primitive fifth root of unity. Then the restriction map $H^1(G_\Q,\Z/5\Z) \to H^1(G_K,\Z/5\Z)$ is injective since the kernel of this map has exponent dividing both $[K:\Q]=4$ and $\# \Z/5\Z$. Since $A[\varphi], A[\rho]$ and $B[\psi]$ are isomorphic to $\bm_5$, $\bm_5 \times \bm_5$ and $\bm_5$ over $K$ we obtain the following commutative diagram with embeddings as vertical maps.
\[\xymatrix{ 0 \ar[r]& \coker  \varphi_\Q \ar[r]^{\psi}\ar[d]& \coker \rho_\Q \ar[r]\ar[d]& \coker \psi_\Q \ar[r]\ar[d] &0\\
 0 \ar[r]& K^*/K^{*5} \ar[r] & (K^*/K^{*5} )^2 \ar[r]&K^*/K^{*5} \ar[r] &0
 }
\]
As above the lower second horizontal map is just $(x,y)\mapsto x^n/y$. Hence to determine $\coker \varphi_\Q$ it suffices to determine the kernel of $x^n/y$ on $\coker \eta_{1,\Q} \times \coker \eta_{2,\Q} \to \coker \psi_\Q$. Again this is independent of $n$. We do this as follows:

\begin{enumerate}
\item For some $\tilde d \in K$ there is a $K$-isomorphism $\tau: E'_{d} \to E_{\tilde d}$, sending a generator of $\ker \eta^\vee$ to $(0,0)$. The map $f:E'_{d} \to K^*/K^{*5}$ is then $P\mapsto -x(\tau(P))^2+y(\tau(P))+x(\tau(P))y(\tau(P))$. Hence we have to determine $\tau$. This can be done easily for each individual curve $E'_{d}$.
\item To represent elements in $\coker \eta_{\Q} \subset K^*/K^{*5}$ note that the class number of $K^*$ equals $1$. Set
$$K(S,5):=\{x \in K^*/K^{*5} \mid \ v_\p(x)\equiv 0 \bmod 5, \ \forall \p \notin S\},$$
where $S$ contains all primes $\p$ of $K$ being a bad prime of $E_{d}$ or dividing $5$, i.e., all primes $\p$ of $K$ lying over a primes $p$, such that $p\mid 5uv(u^2+11uv-v^2)$. From \cite[Exercise 10.9]{SilvermanI} it follows that $f(\coker \eta_\Q) \subset K(S,5)$. Hence to represent elements in $\coker \eta_{\Q}$ we have to fix a generator $t_\p$ for each prime $\p \in S$, and we have to fix generators for the unit group of $K$ modulo fifth powers. The field $K$ is well-known and it is easy to see that the unit group is generated by $-\zeta_5$ and $(1+\zeta_5)$.
Hence we can write 
$$f(P)\equiv \zeta_5^{a_0}(1+\zeta_5)^{a_1}\prod_{\p\in S} t_\p^{v_\p(f(P))}$$ 
modulo fifth powers.
\end{enumerate}

\begin{remark}
We can weaken the assumption of having a basis for the Mordell-Weil groups $E_{d_1}(\Q), E_{d_2}(\Q),E'_{d_1}(\Q)$ and $E'_{d_2}(\Q)$. It is actually sufficient to just have generators of a finite index sublattice of these four groups, such that the index is not divisible by $5$, i.e., the generators of infinite order are not divisible by $5$ modulo torsion. This is the case, since the images of such sublattices in the co-kernels of $\eta_i^\vee$, respectively $\eta_i$, are the complete co-kernels. Also it is sufficient to just know such sublattices of $E_{d_1}(\Q)$ and $E_{d_2}(\Q)$, since suitable dual sublattices can be easily computed using the isogenies $\eta_i$. One only has to calculate the images of the generators under $\eta_i$ and then check if their span contains points divisible by $5$ modulo torsion.
\end{remark}

\section{Determining the local factor}\label{SecLoc}

We want to calculate   $\frac{\#\coker \varphi_{\Q_p}}{\#\ker \varphi_{\Q_p}}$ for all bad primes $p$ and for $p= \deg \varphi = 5$. Since the kernel of $\varphi_{\Q_p}$ is generated by a $\Q$-rational point it follows that $\#\ker \varphi_{\Q_p}=5$. The size of the co-kernel of $\varphi_{\Q_p}$ depends on the reduction of $E_{d_1}$ and $E_{d_2}$, but turns out to be independent of $n$.

For $\eta:=\eta_i$, we first describe how $\coker \eta_{\Q_p}$ depends on the reduction type of $E:=E_{d_i}$. Write $d_i=:u/v$ with $u,v \in \Z$ and $\gcd(u,v)=1$. Then $E$ has the following global minimal equation
\[ E: y^2+(u+v)xy+uvy=x^3+uv^2x^2\]
and discriminant $-(uv)^5(u^2+11uv-v^2)$. 

\begin{lemma} $E$ has the following reduction type at a prime $p$:
 \begin{enumerate}
  \item If $p\mid uv$ then the reduction is split multiplicative and the point $(0,0)$ does not lie on the identity component of the N\'eron model of $E$.
\item If $p\mid u^2+11uv-v^2$ then $(0,0)$ lies on the identity component of the N\'eron model of $E$ and either $p=5$, or $p\equiv \pm 1 \bmod 5$ holds. If $p=5$ the reduction is additive, if $p\equiv 1 \bmod 5$ then the reduction is split multiplicative, and if $p\equiv 4\bmod 5$ then the reduction type is non-split multiplicative. 
 \end{enumerate}
\end{lemma}
\begin{proof}
Let $\overline{E}$ be $E \bmod p$ and $\overline{E}_{\ns}$ be the smooth locus of $\overline{E}$. If $p\mid uv$ then $\overline{E}$ has equation $y^2+\alpha xy=x^3$, for some non-zero $\alpha \in \Z/p\Z$. In particular, $(0,0) \bmod p$ is a node of $\overline{E}$ and the tangent cone is generated by $x=-\alpha y$ and $y=0$, hence the reduction is split multiplicative. Since $(0,0)$ reduces to the singular point of $\overline{E}$ this point does  not lie on the identity component of the N\'eron model of $E$.

If $p \mid u^2+11uv-v^2$ then the reduction of $(0,0)$ is both on $\overline{E}_{\ns}$ and is non-trivial. In particular the order of the reduction of $(0,0)$, which is $5$, divides $\#\overline{E}_{\ns}(\F_p)$. If the reduction is split multiplicative then this group has order $p-1$, if the reduction is non-split then this group has order $p+1$, and if the reduction is additive then this group has order $p$, i.e., $p\equiv 1 \bmod 5$, $p\equiv -1\bmod 5$, and $p=5$ respectively.
\end{proof}

Let $E':=E'_{d_i}$ be the isogenous elliptic curve. Denote by $c_{E,p}$ and $c_{E',p}$ the local Tamagawa numbers, i.e., the number of components of the N\'eron model. 
\begin{lemma}  For the Tamagawa quotient we have
$$\frac{c_{E',p}}{c_{E,p}} = \begin{cases}
\frac{1}{5}, & \text{ if } p\mid uv,\\
5,& \text{ if } p\mid u^2+11uv-v^2 \text{ and } p\equiv 1 \bmod 5,\\
1, & \text{ otherwise}.
\end{cases}$$
\end{lemma}
\begin{proof}
Since $\eta$ has degree $5$ it follows that that $\frac{c_{E',p}}{c_{E,p}} =5^a$ for some $a \in \Z$. If the reduction is different from split multiplicative then $c_{E,p}$ and $c_{E',p}$ are at most 4, hence $a=0$ and $c_{E,p}=c_{E',p}$.

In \cite[Proposition 2.16]{Keil} it is shown by using Tate curves that if the reduction is split multiplicative then $a\in \{-1,1\}$, depending on whether the kernel is on the identity component of the N\'eron model or not.
\end{proof}

If $p\nmid \deg \eta =5$ then from \cite[Lemma 3.8]{schaefer} it follows that $\frac{\#\coker \eta_{\Q_p}}{\#\ker \eta_{\Q_p}}=\frac{c_{E',p}}{c_{E,p}}$.
Using this it follows easily that 
\begin{lemma} \label{lem:coker_eta_p} Suppose $p$ is a prime different from $5$.
\begin{enumerate}
 \item If $p\mid u^2+11uv-v^2$ and $p\equiv 4 \bmod 5$, then $\coker \eta_{\Q_p}=\Z/5\Z$.
\item  If $p\mid u^2+11uv-v^2$ and $p\equiv 1 \bmod 5$, then $\coker \eta_{\Q_p}=(\Z/5\Z)^2$.
\item If $p\mid uv$, then $\coker \eta_{\Q_p}=0$.
\item If $p$ is good for $E$, then $\coker \eta_{\Q_p}=\Z/5\Z$.
\end{enumerate}
\end{lemma}
Now $\coker \eta_{\Q_p}\subset H^1(G_{\Q_p},\Z/5\Z)$. From \cite[Section II.5 Theorem 2, Proposition 17]{serre-gal-coh} it follows that for $p \nmid \deg \eta = 5$
\[ \#H^1(G_{\Q_p},\Z/5\Z)=\#H^0(G_{\Q_p},\Z/5\Z)\#H^0(G_{\Q_p},\bm_5)=5^a,\]
with $a=1$, if $p\equiv 4 \bmod 5$, and $a=2$, if $p\equiv 1 \bmod 5$. From this it follows that

\begin{proposition} Suppose $p$ is a prime of bad reduction for $E$ and $p\neq 5$ with $p\mid u^2+11uv-v^2$. Then $\coker \eta_{\Q_p}=H^1(G_{\Q_p},\Z/5\Z)$.
\end{proposition}

We now return to our abelian surface $A$. Then the above proposition enables us to determine $\coker \varphi_{\Q_p}$ for bad primes different from $5$. 

\begin{proposition} Suppose $p$ is a prime of bad reduction for $A$, $p\neq 5$. Then $\coker \varphi_{\Q_p}$ is isomorphic as an abelian group to
$$\begin{cases}
0, & \text{ if } p\mid u_1v_1u_2v_2,\\
(\Z/5\Z)^2, & \text{ if } p\mid \gcd(u_1^2+11u_1v_1-v_1^2,u_2^2+11u_2v_2-v_2^2), \ p\equiv 1 \bmod 5,\\
\Z/5\Z, & \text{ otherwise}.
\end{cases}$$
\end{proposition}
\begin{proof}
Recall that $\coker \varphi_{\Q_p} = \ker \left( \coker \eta_{1,\Q_p} \times \coker \eta_{2,\Q_p} \to \coker \psi_{\Q_p} \right),$ which equals
\[ \left( \coker \eta_{1,\Q_p} \times \coker \eta_{2,\Q_p} \right) \cap \ker \left(H^1(G_{\Q_p},\Z/5\Z)^2 \to H^1(G_{\Q_p},\Z/5\Z) \right).\]
The surjective map $H^1(G_{\Q_p},\Z/5\Z)^2 \to H^1(G_{\Q_p},\Z/5\Z)$ is given by $(x,y)\mapsto nx-y$. Suppose that $p\mid u_1v_1u_2v_2$, then by Lemma \ref{lem:coker_eta_p} we have that $\coker \eta_{i,\Q_p}=0$, for at least one $i$, and therefore $\coker \varphi_{\Q_p}=0$.

Suppose now $p\nmid u_1v_1u_2v_2$. By assumption one of the $E_{d_i}$ has bad reduction at $p$, let's say $E_{d_1}$. Since $p\nmid 5u_1v_1$ it follows from the above proposition that $\coker \eta_{1,\Q_p}=H^1(G_{\Q_p},\Z/5\Z)$ and hence $\coker \varphi_{\Q_p}\cong \coker \eta_{2,\Q_p}$. Now $E_{d_2}$ has either additive or good reduction. The reduction of $E_{d_2}$ is additive if and only if $p\mid \gcd(u_1^2+11u_1v_1-v_1^2,u_2^2+11u_2v_2-v_2^2)$. Now apply Lemma \ref{lem:coker_eta_p} to deduce the structure of $\coker \eta_{2,\Q_p}$, hence the structure of $\coker \varphi_{\Q_p}$. \end{proof}

It remains to check the case $p=5$. As before, we first have a look at the elliptic curve $E$. If $5 \mid uv$ then as above the reduction is split multiplicative and $\frac{c_{E',p}}{c_{E,p}}=\frac{1}{5}$. Using Tate curves one easily shows that $\coker \eta_{\Q_p}=0$.

If $5\mid u^2+11uv-v^2$ then the reduction is additive. In particular the component groups of $E$ and $E'$ have the same order, which is also the case if the reduction is good. Therefore $\frac{c_{E',p}}{c_{E,p}}=1$. The isogeny $\eta:E \to E'$ can be written as a power series in one variable in a neighbourhood of the point $O$. Again from \cite[Lemma 3.8]{schaefer} it follows that
\[ 
 \frac{\#\coker \eta_{\Q_5}}{\#\ker \eta_{\Q_5}} = |\eta'(0)|_5^{-1},
\]
where $|\eta'(0)|_5$ is the normalized $5$-adic absolute value of the leading coefficient of the power series representation of $\eta$ evaluated at $0$. This can be easily computed using V\'elu's algorithm \cite{velu}. In \cite[Lemma 4.1, Proposition 4.2]{Keil} it is shown that in the additive case $v_5(u^2+11uv-v^2)\in \{2,3\}$ and that if $v_5(u^2+11uv-v^2)=2$ then $|\eta'(0)|_5=1$, and if $v_5(u^2+11uv-v^2)=3$ then $|\eta'(0)|_5=1/5$. If $E$ has good reduction at $p=5$ then it follows that $\# \coker \eta_{\Q_p}=\# \ker \eta_{\Q_p}$, because in this case we also have $|\eta'(0)|_5=1$. We summarize as follows.

\begin{lemma} Suppose $p=5$.
\begin{enumerate}
\item If $p\mid uv$, then $\coker \eta_{\Q_p}=0$.
\item If $p$ is good for $E$, then $\coker \eta_{\Q_p}=\Z/5\Z$.
 \item If $p\mid u^2+11uv-v^2$, then $\coker \eta_{\Q_p}= \begin{cases}
(\Z/5\Z)^2, & 5^3 \mid u^2+11uv-v^2,\\
\Z/5\Z, &  \textnormal{otherwise.}
\end{cases}$
\end{enumerate}
\end{lemma}

Now we can calculate $\coker \varphi_{\Q_p}$ in the remaining case $p=5$.

\begin{lemma} The cardinality of $\coker \varphi_{\Q_5}$ equals the cardinality of $\ker \varphi_{\Q_5}$, unless
 \begin{enumerate}
  \item $5\mid u_1v_1u_2v_2$. In this case $\#\coker \varphi_{\Q_5}=0$.
\item $5^3 \mid \gcd(u_1^2+11u_1v_1-v_1^2,u_2^2+11u_2v_2-v_2^2)$. In this case $\# \coker \varphi_{\Q_5}=5^2$.
 \end{enumerate}

\end{lemma}

\begin{proof} If $\coker \eta_{i,\Q_5}=0$, for one $i$, then $\coker \varphi_{\Q_5}=0$. The first condition is equivalent with $5\mid u_1v_1u_2v_2$.

Suppose now that $\coker \eta_{i,\Q_5}\neq 0$ for both $i$, which implies that $p=5$ is additive or good for $E_{d_i}$. We need two facts from \cite[Section II.5 Proposition 18, Theorem 5]{serre-gal-coh}, namely $H^1(G_{\Q_5},\Z/5\Z)=(\Z/5\Z)^2$ and $H^1_{nr}(G_{\Q_5},\Z/5\Z)=\Z/5\Z$. As in the previous proposition we have that if $\coker \eta_{1,\Q_5}=H^1(G_{\Q_5},\Z/5\Z)$, then $\coker \varphi_{\Q_5} \cong \coker \eta_{2,\Q_5}$ and vice versa. This gives the second case of the lemma, since $\coker \eta_{i,\Q_5}=(\Z/5\Z)^2$ if and only if $5^3 \mid u_i^2+11u_iv_i-v_i^2$, and $\coker \eta_{i,\Q_5}=\Z/5\Z$ otherwise. 

It remains to consider $\coker \eta_{1,\Q_5}=\coker \eta_{2,\Q_5}=(\Z/5\Z)$. In this case one can show that $\coker \eta_{i,\Q_5} = H^1_{nr}(G_{\Q_5},\Z/5\Z)$, for both $i$; see \cite[Proposition 2.10, Proposition 3.5]{Keil} and \cite[Section 3]{schaefer-stoll}. Thus the kernel of $ \coker \eta_{1,\Q_5} \times \coker \eta_{2,\Q_5} \to \coker \psi_{\Q_5}$, which equals $\coker \varphi_{\Q_5}$, has five elements. This finishes the proof.
\end{proof}

Putting everything together yields

\begin{proposition} Let $p$ be a prime. Then 
\[ \frac{\#\coker \varphi_{\Q_p} }{\# \ker \varphi_{\Q_p}}\]
 is a non-square if and only if one of the following occurs
\begin{enumerate}
 \item $p\mid u_1v_1u_2v_2$,
\item $p\mid \gcd (u_1^2+11u_1v_1-v_1^2,u_2^2+11u_2v_2-v_2^2)$ and $p\equiv 1 \bmod 5$,
\item $p^3\mid \gcd (u_1^2+11u_1v_1-v_1^2,u_2^2+11u_2v_2-v_2^2)$ and $p=5$.
\end{enumerate}
\end{proposition}

\section{Algorithm}\label{SecAlg}
The code was implemented in Sage \cite{sage} and is available at \cite{worksheet}. The algorithm consists of two main steps and an initialization step, which we call step 0. In step 1 one creates a database of elliptic curves having a point $P$ of order $5$, which are parametrized by two coprime positive integers $(u,v)$. One has to specify which pairs $(u,v)$ one wants to consider. In step 2 one takes such a database of elliptic curves $E_d$, for $d=u/v$, and goes over all pairs of these curves and determines whether the order of the Tate-Shafarevich group of the abelian surfaces $B=E_{d_1} \times E_{d_2} / \langle (P_1,P_2)\rangle$ is a square. For trivial reasons, pairs of the same elliptic curve are omitted and pairs are considered to be without order.

Step 0: Fix a (large) integer $M$. For each prime number $p\leq M$ determine the 
prime ideals $\p$ of  $K = \Q(\zeta_5)$ above $p$ and fix an ordering of them. Then fix for each prime ideal $\p$ a generator $t_\p$.

Step 1: Fix a positive integer $N$. For each pair of coprime positive integers $(u,v)$, such that $\max(u,v)\leq N$, we collect the following data associated to $E:=E_d$, for $d=u/v$. (Optional: Filter the pairs $(u,v)$ by other limitations, e.g., considering only those for which the corresponding elliptic curves have conductor $\leq C$.)
\begin{itemize} 
\item Collect all the primes dividing $5uv(u^2+11uv-v^2)$ in a set $S$.
\item Collect all the primes dividing $uv$ in a set $T$.
\item Collect all the primes $p\equiv 1\bmod 5$ dividing $u^2+11uv-v^2$ in a set $U$.
\item If $v_5(u^2+11uv-v^2)=3$, put also $p=5$ into the set $U$.
\item Determine the analytic rank $r$ of $E$.
\item Determine a system of $r$ generators of a sublattice $\Lambda$ of $E(\Q)$, such that the points of infinite order modulo torsion are not divisible by 5. Take the image of $\Lambda$ in $\Q(S,5)$ to determine a basis $P$ of $\coker \eta_\Q^\vee\subset \Q(S,5)$. The data for each basis element consists of a pair for each prime in $S$, where the first entry is the corresponding element in $S$ and the second entry is the exponent as an element in $\Z/5\Z$. 
\item Calculate the image of $\Lambda$ under $\eta$ in $E'(\Q)$ and determine which image points are divisible by 5 modulo torsion. Divide if possible and determine the non-trivial 5-torsion points of $E'(\Q)$ to get a sublattice $\Lambda '$ of $E'(\Q)$, such that the points of infinite order modulo torsion are not divisible by 5. Use this information to get $\dim \coker \eta_\Q$. 
\item Take the image of $\Lambda '$ in $K(S,5)$ to determine a basis $Q$ for $\coker \eta_\Q \subset K(S,5)$. The data for each basis element consists of a pair for each prime in $S$ and a pair for the units, where the first entry is the corresponding element in $S$, respectively $1$, and the second entry is a list of elements in $\Z/5\Z$, which contains as many entries as there are prime ideals $\p$ in $K$ over $p$, respectively two entries, and these entries are the exponents corresponding to the prime ideals $(t_\p)$ with the chosen order, respectively the exponents of the units. 
\end{itemize}

Step 2: To determine whether the order of $\sha(B_{d_1,d_2}/\Q)$ is a square, for each pair of pairs $(u_1,v_1),(u_2,v_2)$ from the first step (modulo ordering and equality), do the following:
\begin{itemize} 
 \item Set $L:=-\# (T_1\cup T_2)+\# (U_1\cap U_2)$.
\item Fix an ordering for $\mathcal{S}:=S_1 \cup S_2$.
\item Write out the elements from $P_1 \cup P_2$ into a matrix with respect to $\mathcal{S}$. This gives a matrix with entries in $\Z/5\Z$. Calculate the rank of this matrix, which equals the dimension of $\coker \varphi^\vee_\Q$.
\item Write out the elements from $Q_1 \cup Q_2$ into a matrix with respect to the prime ideals $(t_\p)$ lying over the primes of $\mathcal{S}$ (and with respect to the units). This gives a matrix with entries in $\Z/5\Z$. Calculate the rank of this matrix, which equals the dimension of $\coker \psi_\Q$. 
\item Set $G:= \dim \coker \varphi^\vee_\Q - \dim \coker \eta_{1,\Q} - \dim \coker \eta_{2,\Q} + \dim \coker \psi_\Q$. (Recall that $\dim \coker \varphi_\Q = \dim \coker \eta_{1,\Q} + \dim \coker \eta_{2,\Q} - \coker \psi_\Q$.)
\end{itemize}

Then the local factor (without the infinite prime) is a non-square if and only if $L$ is odd, and the global factor (without the kernels) is a non-square if and only if $G$ is odd. Since the contribution of the infinite prime and the kernels cancel, we have that $\sha (B_{d_1,d_2}/\Q)$ has non-square order if and only if $L+G$ is odd.

The constructed databases and obtained results are given in the next section. To conclude this section, we make some comments on the implementation. Step 0 in the cases considered is not computational demanding. For example, on a desktop computer it may take some seconds up to a few minutes to compute all generators for all prime ideals lying over all primes up to $500,000$. Step 2 is also no problem. It consists only of simple set operations and calculating ranks of small matrices with coefficients in $\Z/5\Z$. Even a few million of pairs of elliptic curves can be considered in under an hour. 

The computational demanding part is step 1. There are two main issues. The most problematic calculation is determining $r$ generators of a finite index subgroup of the Mordell-Weil group, where $r$ is the analytic rank. We used the standard Sage method `E.point\_search(height\_limit=18,rank\_bound=r)', and in case this did not come up with enough points we tried some of the remaining curves with `E.gens()'. In several cases these methods did not provide an answer within 48 hours on a single CPU for a single elliptic curve. For these curves we used the method `MordellWeilShaInformation()' in MAGMA \cite{magma}, which could handle all our problematic curves in a few seconds each. 

The second problematic calculation in the actual code is computing the image of $\coker \eta_\Q$ in $K(S,5)$. We try to factor ideals of $K$, which are generated by elements of possibly very big norm. For example, the curve $E_d$, for $d=1/94$, has analytic rank $1$ and the numerator and denominator of the image of the point of infinite order in $K(S,5)$ have about $600$ digits and Sage was not able to factor the corresponding ideal. As we already knew that the image is trivial, since the dimension of $\coker \eta_\Q$ was zero, we could skip this calculation. Considering this information in the algorithm all curves we tried worked fine. This problem might be avoidable by trying another strategy working modulo primes. The rest of step 1 is no problem for moderately chosen $d=u/v$, since it is mainly prime factorization of integers and of rational polynomials of degree $25$ (to divide points by $5$), as well as calculating isogenies and analytic ranks. On a desktop computer one could produce in a few hours a database of a few thousand curves.

\begin{remark}
Step 0 and 2 do not use any assumptions, but for step 1 we assume the Birch and Swinnerton-Dyer conjecture in case the elliptic curve is of analytic rank $r \geq 2$, to conclude that the calculated sublattices $\Lambda$ and $\Lambda'$ are of finite index and that the Tate-Shafarevich groups are finite. Thus, only in case that both elliptic curves $E_{d_i}$ have analytic rank $r \leq 1$ the result of the algorithm about $\# \sha(B_{d_1,d_2})$ modulo squares is completely unconditional. 
\end{remark}

\section{Results}\label{SecRes}

Given the above described algorithm, one can produce in short time millions of examples of abelian surfaces over $\Q$, such that either the order of the Tate-Shafarevich group is a square or five times a square. In case the two elliptic curves were both of analytic rank $r \leq 1$ these examples are completely unconditional. We constructed two databases of elliptic curves using step 1 of the algorithm. The first database consists of all elliptic curves $E_d$, $d=u/v$, with $u,v$ positive integers and $\max(u,v)\leq 50,000$, such that the conductor of $E_d$ is bounded by $C=10^6$. The second database consists of all elliptic curves, such that $\max(u,v) \leq 100$, for $u,v$ positive integers. 

Database 1 contains $2212$ elliptic curves, all of them having analytic rank $r \leq 2$. It is likely that there are no further elliptic curves of conductor at most $10^6$, which have  a rational torsion point of order $5$, since there is no such curve with $4617 < \max(u,v) \leq 50,000$. The database is described in more detail in Table 1, where we state for each analytic rank the number of elliptic curves with conductor at most $10^6$ and with $\max(u,v)\leq N$. Database 2 contains $6,087$ elliptic curves. All of them have analytic rank $r \leq 3$. See Table 2 for more details. In the following we will present the results of step 2 of the algorithm applied to the two databases described above. 

Database 1 yields $2,445,366$ abelian surfaces $B_{d_1,d_2}$. It turns out that $47.01\%$ of these surfaces have a Tate-Shafarevich group of non-square order. Database 2 leads to $18,522,741$ abelian surfaces. The percentage of the non-square case is $49.31$. The intersection of the two databases consists of $1,391$ curves, hence we considered $966,745$ surfaces twice. In total this gives $20,001,362$ surfaces, of which $49.16\%$ have a Tate-Shafarevich group of non-square order.

\begin{center}
\begin{table}[bhtp]
    \begin{tabular}{|r|r|r|r|r|}
    \hline
    $N$&
    $\# E_d$, $C= 10^6$&
    $\# \{r=0\}$&
    $\# \{r=1\}$&
    $\# \{r=2\}$   
    \\
    \hline
  50,000  & 2,212  & 987  & 1,109 & 116 \\
  4,617  & 2,212  & 987 &1,109  & 116 \\
   3,375 & 2,211 & 986 &1,109  & 116\\
   3,072  & 2,210 & 986 & 1,108  & 116\\
   2,695  & 2,209 & 986  &1,107  & 116 \\
  2,000  & 2,200 & 982  &1,102 & 116 \\
  1,000  & 2,174 & 963 &1,095 & 116\\
   900 & 2,170 & 961 &1,093  & 116 \\
   800 & 2,159  & 956 &1,088 & 115 \\
  700  & 2,145 & 951  &1,079  & 115\\
   600  & 2,119 & 941  &1,063  & 115 \\
 500  & 2,088 &921  &1,052  &115\\
   400 & 2,066 & 912 &1,039 &115 \\
  300 & 1,993  &872  &1,009  &112 \\
   200 & 1,818 & 786 & 929 &103 \\
   100 & 1,391 & 616  &697  &78\\
  50 & 845& 394  &405 & 46 \\
    \hline
    \end{tabular}
  \caption{Database 1: Number of elliptic curves $E_d$ with conductor $\leq 10^6$ and $\max(u,v)\leq N$.}
\end{table}
\end{center}

\begin{center}
\begin{table}[bhtp]
    \begin{tabular}{|r|r|r|r|r|r|}
    \hline
    $N$&
    $\# E_d$&
    $\# \{r=0$\}&
    $\# \{r=1$\}&
    $\# \{r=2$\}&
    $\# \{r=3$\}
    \\
    \hline
 100 & 6,087  & 2,390 & 3,038 & 633 & 26 \\
 90 & 4,959  & 1,987& 2,463  & 490 & 19 \\
 80 & 3,931  & 1,597 & 1,940  & 380 & 14 \\
 70 & 2,987  & 1,235 & 1,455  & 287 & 10 \\
 60 & 2,203 & 925& 1,074  & 198& 6\\
 50 & 1,547 & 660 & 760 & 123& 4 \\
 40 & 979 & 412 & 494 & 70 & 3 \\
 30 & 555& 245 & 277 & 33& - \\
 20 & 255& 130& 115 & 10& - \\
 10 & 63& 40& 22& 1& -\\
    \hline
    \end{tabular}
  \caption{Database 2: Number of elliptic curves $E_d$, such that  $\max(u,v)\leq N$.}
\end{table}
\end{center}

We did two different experiments with the two databases. In experiment 1 we investigated how the rank influences the squareness of the Tate-Shafarevich group. We list the result in Table 3 for Database 1 and in Table 4 for Database 2. The first three, respectively four, entries correspond to pairs $(E_1,E_2)$ with the same analytic rank. The following three, respectively six, lines correspond to pairs with different analytic ranks and the final line corresponds to pairs with analytic rank $r \leq 1$. If we consider abelian surfaces of fixed analytic rank of at least $4$ then the density of the surfaces with square Tate-Shafarevich group seems to be significant larger than $0.5$. However the surfaces with rank larger than $2$ inside our family are conjectured to have density zero and our database contains very few such cases. The calculations with curves of rank $r\leq 1$ all show that the non-square case happens in about $50\%$ of all cases. For both experiments we list how many abelian surfaces $B_{d_1,d_2}$ occur in each of the cases, we state the percentage of the surfaces with square Tate-Shafarevich group, and we give the percentage of in how many cases the parity of the rank of the abelian surface agrees with the parity of the exponent of the regulator quotient (RE). Note that the results are unconditional in case $\textnormal{rk}(E_i)\leq 1$, for both $E_i$. If one of the analytic ranks is at least $2$ then we assume that it is at least the Mordell-Weil rank.

\begin{center}
\begin{table}[bhtp]
    \begin{tabular}{|l||r|r|r|}
    \hline
    & \multicolumn{3}{|c|}{$C=10^6, \ N=50,000$}\\
    \hline
    &$\# B_{d_1,d_2}$& $\% \sha=\square$ & \% RE $\equiv$ rk $(2)$  \\
    \hline
$r=0$   & 486,591 & 54.041& 100.00\\
$r=1$ & 614,386  & 58.614 & 63.51\\
$r=2$ &  6,670 &  92.039&55.53  \\
\hline
$r=0, r=1$ & 1,094,583  & 46.634 & 83.44 \\
$r=0, r=2$    & 114,492  &  52.867&47.96 \\
$r=1, r=2$   & 128,644   & 74.314 &42.48 \\
\hline
$r\leq 1$ & 2,195,560  &  51.628& 81.53 \\
\hline
 \end{tabular}
  \caption{Results of experiment 1 for Database 1.}
\end{table}
\end{center}

\begin{center}
\begin{table}[bhtp]
    \begin{tabular}{|l||r|r|r|}
    \hline
    & \multicolumn{3}{|c|}{$N=100$} \\
    \hline
    &$\# B_{d_1,d_2}$& $\% \sha=\square$ & \% RE $\equiv$ rk $(2)$ \\
    \hline
$r=0$   & 2,854,855   & 48.598 & 100.00\\
$r=1$ & 4,613,203  & 48.882 & 80.91\\
$r=2$ & 200,028  & 73.031 &44.03 \\
$r=3$& 325 &  98.154& 51.08 \\
\hline
$r=0, r=1$  & 7,260,820  & 51.366 & 91.02\\
$r=0, r=2$  & 1,512,870  & 50.567 & 71.36\\
$r=0, r=3$   & 62,140  & 49.891 & 52.73\\
$r=1, r=2$   & 1,923,054  & 52.717 & 59.50\\
$r=1, r=3$  & 78,988  & 60.632 & 46.23\\
$r=2, r=3$   & 16,458  & 84.470 & 48.23\\
\hline
$r\leq 1$ & 14,728,878  & 50.051 & 89.59\\
\hline
 \end{tabular}
  \caption{Results of experiment 1 for Database 2.}
\end{table}
\end{center}

In experiment 2 we looked for the behaviour of the distribution of square and non-square order Tate-Shafarevich groups for increasing conductor, respectively height, of the elliptic curves. Hence we filtered Database 1 for different values of conductor bounds $C$ and Database 2 for different values of height bounds $N$. For low bounds, the non-square case was less likely. When we increase these bounds this frequency tends to approximately $50\%$. The results of experiment 2 is given in Table 5 for Database 1 and Table 6 for Database 2. Note that for some of the surfaces we assume the weak form of the Birch and Swinnerton-Dyer conjecture mentioned above.

\begin{center}
\begin{table}[bhtp]
    \begin{tabular}{|r||r|r||r|r|}
    \hline
    $C$&
    $\# E_d$&
    $\# B_{d_1,d_2}$&
    $\% \sha=\square$ & \% RE $\equiv$ rk $(2)$\\ 
    \hline
  1,000,000  & 2,212 & 2,445,366 &52.990 & 77.84\\
  800,000 & 1,966 & 1,931,595 &53.232 & 77.16\\
600,000  & 1,683   & 1,415,403 & 53.758 & 76.06\\
 400,000  & 1,351  & 911,925  & 54.215 & 75.24\\
  200,000 & 924  &  426,426  &55.001 & 73.91 \\
  100,000 & 623  &193,753  &57.074 & 74.29\\
  80,000 & 547 &149,331  &57.776  & 74.03 \\
60,000  & 470  &110,215  & 57.990 &72.75 \\
 40,000  & 376  & 70,500  & 59.306 & 73.34\\
  20,000 & 245  & 29,890  & 61.288 & 71.72\\
  10,000 & 152  & 11,476 & 62.182 & 72.59 \\
  5,000 & 110 &  5,995  & 59.783 & 71.79\\
  1,000 & 45 &  990 & 65.556 & 76.77\\
    \hline
    \end{tabular}
  \caption{Results of experiment 2 for Database 1.}
\end{table}
\end{center}

\begin{center}
\begin{table}[bhtp]
    \begin{tabular}{|r||r|r||r|r|}
    \hline
    $N$&
    $\# E_d$&
    $\# B_{d_1,d_2}$&
    $\% \sha=\square$ & \% RE $\equiv$ rk $(2)$ 
    \\
    \hline
  100 & 6,087 & 18,522,741  & 50.694 &84.14 \\
  90 & 4,959 &  12,293,361  & 50.821 &83.66 \\
80  & 3,931 & 7,724,415   & 50.941 & 83.32\\
70  & 2,987 & 4,459,591   & 51.235 &82.51 \\
 60  & 2,203 & 2,425,503  & 51.461 & 82.00 \\
 50 & 1,547 & 1,195,831  &52.211 & 80.85\\
  40 & 979 & 478,731  & 52.764 & 79.92\\
  30 & 555 & 153,735  & 54.157 & 77.12\\
  20  & 255 & 32,385  &56.384  &77.11 \\
  10 & 63 &  1,953  & 67.179 & 74.04\\
    \hline
    \end{tabular}
  \caption{Results of experiment 2 for Database 2.}
\end{table}
\end{center}

The two ways we ordered the elliptic curves, via conductor or via height, are natural ways of ordering elliptic curves. It is conjectured that densities obtained concerning these orderings agree. In both cases the densities seem to exist and are around 0.5. This is in contrast to the results of Poonen and Stoll \cite{PS}, who showed that for genus $2$ curves the density of the non-square Jacobians is about $0.13$ and this density tends to zero, as the genus goes to infinity.

We end by giving some heuristics why we expect the density to be $50\%$. We expect that for a random pair $(d_1=u_1/v_1,d_2=u_2/v_2)$ in $\Q^* \times \Q^*$ the global factor is a square for $50\%$ of the abelian surfaces and that the local factor is a square for $50\%$ of them, too. We also expect these distributions to be independent. Using the $18,522,741$ pairs obtained from the second database, we get numerical evidence for the independence, as illustrated in the following table.

\begin{center}
    \begin{tabular}{|c|c|c|}
    \hline
    $N=100$   & global quotient $=\square$ & global quotient $\neq \square$\\
    \hline
    local quotient $= \square$ & $26.08\%$  & $25.26\%$  \\
    \hline
    local quotient $\neq \square$ & $24.04\%$ & $24.61\%$ \\
    \hline
    \end{tabular}
\end{center}

Recall that the exponent of the local quotient equals $-\# (T_1\cup T_2)+\# (U_1\cap U_2)$, hence one could prove the expected densities for the local quotient by showing that the probability that the set $(T_1\cup T_2)$ has an even number of elements is independent of the probability that the set $(U_1\cap U_2)$ has an even number of elements. The corresponding numerical result for Database 2 is gathered in the following table.

\begin{center}
    \begin{tabular}{|c|c|c|}
    \hline
    $N=100$   & $\# (U_1\cap U_2) \equiv 0 \ (2)$  & $\# (U_1\cap U_2) \equiv 1 \ (2)$\\
    \hline
    $\# (T_1\cup T_2)\equiv 0 \ (2)$ & $46.71\%$  & $1.80\%$  \\
    \hline
    $\# (T_1\cup T_2) \equiv 1 \ (2)$ & $49.55\%$ & $1.95\%$ \\
    \hline
    \end{tabular}
\end{center}

The global quotient is harder to control. The exponent of the torsion quotient equals $3$ on a density $1$ subset of the pairs $(d_1,d_2)$, see \cite[Proposition 4.6]{Keil}. The results of Tables 3-6 suggest that the squareness of the regular quotient and hence the squareness of the global quotient is not independent of the parity of the rank. If both ranks of the elliptic curves $E_1$ and $E_2$ are equal to $0$, hence are even, the regulator quotient equals $1$, hence is a square. If one elliptic curve is of rank $0$ and the other is of rank $1$, then the regulator quotient is a non-square if and only if $\coker \eta_\Q$ is trivial modulo torsion, where $\eta$ is the usual isogeny belonging to the elliptic curve of rank $1$. In Database 2 we have the following situation. For the rank $1$ curves it happens in about $91.2\%$ of the cases that $\eta_\Q$ is surjective on the free part. In case both ranks are equal to $1$, the regulator quotient is a square in about $80.9\%$ of the cases. For the complete second database we get that the parity of the exponent of the regulator quotient agrees with the parity of the rank in $84.14\%$. If we consider only all the elliptic curves of rank $\leq 1$, then we have that for abelian surfaces $B_{d_1,d_2}$ of even rank the regulator quotient is a square in about $88.2\%$ of the cases, and for abelian surfaces $B_{d_1,d_2}$ of odd rank the regulator quotient is a non-square in about $91.0\%$ of the cases, which together gives $89.6\%$ of agreement. We have the following situation for the complete Database 2.

\begin{center}
    \begin{tabular}{|c|c|c|}
    \hline
    $N=100$   & $\textnormal{rk}(B_{d_1,d_2}) \equiv 0 \ (2)$  & $\textnormal{rk}(B_{d_1,d_2}) \equiv 1 \ (2)$\\
    \hline
    regulator quotient $= \square$ & $42.067\%$  & $7.931\%$  \\
    \hline
    regulator quotient $\neq \square$ & $7.927\%$ & $42.075\%$ \\
    \hline
    \end{tabular}
\end{center}

In contrast to the global quotient, the squareness of the local quotient seems to be independent of the parity of the rank of the abelian surfaces. Again we give the numerical results for Database 2.

\begin{center}
    \begin{tabular}{|c|c|c|}
    \hline
    $N=100$   & $\textnormal{rk}(B_{d_1,d_2}) \equiv 0 \ (2)$  & $\textnormal{rk}(B_{d_1,d_2}) \equiv 1 \ (2)$\\
    \hline
    local quotient $= \square$ & $25.670\%$  & $25.675\%$  \\
    \hline
    local quotient $\neq \square$ & $24.324\%$ & $24.331\%$ \\
    \hline
    \end{tabular}
\end{center}

\bibliographystyle{plain}
        \bibliography{on_the_density_of_sha_stefan_keil}

\end{document}